\documentclass{amsart}

\usepackage{amsmath,amsthm,amssymb}
\usepackage{graphicx,color,tikz,wrapfig}
\usepackage{comment,url,enumitem}

\newtheorem{thm}{Theorem}[section]
\newtheorem{prop}[thm]{Proposition}
\newtheorem{cor}[thm]{Corollary}
\newtheorem{conj}[thm]{Conjecture}
\newtheorem{lemma}[thm]{Lemma}

\theoremstyle{definition}

\newtheorem{const}[thm]{Construction}
\newtheorem{ex}[thm]{Example}
\DeclareMathOperator{\dist}{dist}
\DeclareMathOperator{\ecc}{ecc}
\DeclareMathOperator{\diam}{d}
\DeclareMathOperator{\rad}{r}
\DeclareMathOperator{\Rand}{R}
\DeclareMathOperator{\brad}{br}

\DeclareMathOperator{\becc}{ecc}
\newcommand{\bct}[1]{\mathfrak{#1}}
\begin{document}

\title[Block eccentricity, radius, and the Randi\'c index]{Block eccentricity, a radius bound, and an application to the Randi\'c index}

\author[M. Doig]{Margaret I. Doig}
\address{Department of Mathematics, Creighton University}
\email{margaretdoig@creighton.edu}
\date{\today}

\thanks{Supported by CURAS Summer Faculty Research Fund}

\begin{abstract}
We propose a framework for thinking about eccentricity in terms of blocks. We extend the familiar definitions of radius and center to blocks and verify that a central block contains all central points. We classify graphs into two types depending upon the relationship between block radius and vertex radius and between central blocks and central vertices; from this we derive a new lower bound on diameter in terms of the diameter of the central block. We also identify a subgraph which respects the block structure of the original graph and realizes the same vertex radius, and we use it to verify that cactus graphs satisfy a conjectured bound between vertex radius and the Randi\'c index, an invariant from mathematical chemistry.
\end{abstract}

\maketitle

\section{Introduction}\label{sec:intro}

We employ a modification of a traditional concept in graph theory (the extension of eccentricity to blocks) to derive new results on classic questions regarding the center of a graph and the relationship between radius and diameter.

%Block eccentricity itself was previously defined (see, e.g., Buckley and Harary?s 1990 textbook on graph centrality), but no one has yet utilized it to extend the concepts of radius and center to blocks, or to investigate their implications for these very classic results on the location of a graph center or the relationship between radius and diameter. 

\subsection*{Origins of Centrality}
The idea of centrality in a graph traces back to the nineteenth century and the work of Jordan. To aid in his classification of small connected graphs up to isomorphism, including their symmetry groups, he defined the version of centrality that we now call a centroid. Only well into the next century did the concept of centrality begin to be embraced as a way to identify vertices which are somehow medial in the graph, with numerous variations crafted for different types of applications, for example, betweenness centrality as a measure of the influence a node has on information flow in a network, or various forms of eigenvector centrality and their use in internet search algorithms.

The structure of the central points is not obvious. Jordan~\cite[Section~2]{jordan1869assemblages} showed that a tree's centroid consists of either one vertex or two adjacent vertices. Some time later, a parallel theorem emerged concerning the center of a tree, and, in 1953, Harary and Norman showed further that the center of any graph lies in a single block~\cite[Lemma 1]{harary1953dissimilarity}. (Note the theorem on the center of a tree was erroneously attributed to Jordan in his 1869 paper by, e.g., Buckley and Norman~\cite[Theorem~2.1]{buckley1990distance}, but Jordan defines ``center'' the way we now define ``centroid'').

\subsection*{Eccentricity and Centrality} 
We refine these results on the relationship between eccentricity and the block structure of a graph. In Section~\ref{sec:defn}, we extend the concept of eccentricity to a block and define block versions of center and radius. In Section~\ref{sec:ecc}, we utilize the relationship among articulation points, geodesics, and distance to prove Theorem~\ref{thm:bcenters}, that central blocks are essentially the same as Harary and Norman's blocks containing all central points, and further that graphs fall into two types, A and B, based on whether there is a unique central block (equivalently, based on whether the block radius is equal to the vertex radius). We also note in Theorem~\ref{thm:whereismax} the location of the periphery with respect to the central block(s).

\subsection*{Radius vs. Diameter}
In Section~\ref{sec:bound}, we apply our concept of block eccentricity to derive a refinement of the traditional bound $\rad \leq \diam \leq 2\rad$: we show that our type A graphs are exactly those which attain the upper bound, and we give a new (sometimes sharper) lower bound for a type B graph in terms of the diameter of its central block.

\subsection*{Radius and the Randi\'c index}

In Section~\ref{sec:subgraph}, we use our concept of a central block to identify a subgraph which realizes the radius of the original graph and whose structure reflects the block structure of the original graph (specifically, the BC-tree of the subgraph is a subgraph of the BC-tree). In Section~\ref{sec:randic}, we apply this subgraph to a question from mathematical chemistry: The Randi\'c index is an invariant studying the branching of a graph, and a long-standing conjecture states that it is bounded below by the radius (with the trivial exception of an even path). We find a recipe for the Randi\'c index when decomposing a graph into articulation components and verify the conjecture for cactus graphs.

\section{Definitions}\label{sec:defn}

For a graph $G$, we let $V(G)$ be the vertex set and $E(G)$ the edge set. We write the edge between adjacent vertices $u$ and $v$ as $[u,v]$. Unless otherwise stated, all graphs are connected.

\subsection{Blocks and the BC-tree}

A special type of vertex we will use extensively is an \emph{articulation point}, also known as a \emph{separating vertex} or \emph{cut vertex}. This is a vertex whose removal would increase the number of components. A graph without articulation points is \emph{nonseparable} or \emph{2-connected}, and a maximal non-separable subgraph is a \emph{block} (sometimes called a \emph{2-connected component}); note each articulation point is in all incident blocks. We denote the set of blocks $B(G)$. 

To describe the block structure, we may draw a new graph $\bct{G}$, called the \emph{block-cutpoint tree}, often shortened to \emph{BC-tree}. We assign one vertex $\bct{a}$ for each articulation point $a \in V(G)$ and one vertex $\bct{B}$ for each block $B \in B(G)$. If $v \in B$, we further assign an edge $[\bct{a},\bct{B}]$. If $H$ is a subgraph of $G$ whose BC-tree is $\bct{H}$, then $\bct{H}$ is a subgraph of $\bct{G}$, and we call $H$ the \emph{shadow} of $\bct{H}$.

We will regularly use the \emph{articulation components} or \emph{branches} at an articulation point $a$, the maximal subgraphs which do not contain $a$ as an articulation point; equivalently, they are the shadows in $G$ of the connected components of $\bct{G} - \bct{a}$ (note that, if we allow disconnected graphs, then it is traditional to require that the articulation point contain $a$, i.e., it must come from the same component as $a$). We extend this definition to a block: if $B$ is a block, the \emph{articulation components at $B$} are the maximal connected subgraphs which do not contain $B$ and do not have any vertex of $B$ as an articulation point; equivalently, they are the shadows in $G$ of the connected components of $\bct{G} - \bct{B}$. 

Note that there may be a complicated relationship between the two types of articulation components. If $a$ is an articulation point on block $B$, then the set of articulation components at $a$ contains one subgraph for each of the blocks incident to $a$. In contrast, the set of articulation components at $B$ splits apart the single component at $a$ which contains $B$ (and removes $B$ from each of its pieces), while it combines the remaining components into one.

\subsection{Eccentricity and its extension to blocks}

Let $\dist(v_1,v_2)$ be the standard distance metric on a graph, in other words, the minimum number of edges in a path between the vertices $v_1$ and $v_2$; such a shortest path may be called a \emph{geodesic}. We extend this idea of distance to blocks: for a vertex $v \in V(G)$ and a block $B \in B(G)$,
\[\dist(v,B) = \dist(B,v) = \min_{u \in V(B)} \dist(u,v)\]
and, for two blocks $B_1, B_2 \in B(G)$,
\[\dist(B_1,B_2) = \min_{v_i \in B_i} \dist(v_1, v_2).\]
We may also borrow the traditional idea of \emph{eccentricity}, or how far away the graph tends from a given vertex $v$,
\[\ecc(v) = \max_{u \in V(G)} \dist(u,v).\]
A vertex of maximal distance from a given vertex $v$ is called an \emph{eccentric vertex} of $v$. We may extend the concept of eccentricity to blocks (and other subgraphs) as well, as Buckley and Harary did \cite[Section 2.4]{buckley1990distance}:
\[\becc(B) = \max_{v \in V(G)} \dist(B,v).\]
Since $\becc(B)$ ranges over all possible vertices in $B$, the eccentricity of the block is bounded above by the eccentricity of its vertices, though it may not be realized by any given vertex (for example, consider a 4-cycle with a leaf attached at each vertex: the cycle has block eccentricity 1 while its vertices have eccentricity 3).

\begin{prop}\label{prop:ecc}
A block's eccentricity is bounded by its vertices' eccentricity:
\[\becc(B) \leq \min_{v \in V(B)} \ecc(v).\]
\end{prop}

Traditionally, a vertex with minimal eccentricity is called a \emph{central point} and the set of them the \emph{center}, and its eccentricity is named the \emph{radius},
\[\rad(G) = \min_{v \in V(G)} \ecc(v).\]
We extend such ideas to blocks. If a block realizes the minimum block eccentricity, we call it a \emph{central block}, and we declare its eccentricity to be the \emph{block radius}:
\[\brad(G) = \min_{B \in B(G)} \becc(B).\]
This value is frequently much smaller than the radius (e.g., the 4-cycle with a leaf at each vertex has the cycle as the central block with block eccentricity 1, while the pendants have block eccentricity 3; the articulation points inherit the maximum of their incident block eccentricities, or 3). As a direct result of Proposition~\ref{prop:ecc},

\begin{prop}\label{prop:br_v_r}
Block radius is bounded by vertex radius:
\[\brad(G) \leq \rad(G).\]
\end{prop}

Graphs divide into two different groups by behavior which are related to the relationship between $\rad(G)$ and $\brad(G)$. We define \emph{type A} to be the graphs where the invariants are equal and \emph{type B} where they are not. Theorem~\ref{thm:bcenters} will show that a type A graph is one with multiple central blocks and a type B graph one with a unique central block.

We would be remiss not to revisit maximal eccentricity as well. The maximal eccentricity of a vertex is the \emph{diameter}, and a vertex realizing it is a \emph{peripheral vertex} (the set of all such is the \emph{periphery}), 
\[\diam(G) = \max_{v \in V(G)}\ecc(v).\]
It is tempting to define block diameter as maximal block eccentricity, but such a definition would not be compatible with vertex diameter, as the block periphery would not necessarily contain the vertex periphery. To the contrary, we could define a peripheral block to be a block containing a peripheral vertex; further motivation for such a definition would be that we could define an upper distance \[\dist^\prime(B,v) = \max_{u \in V(B)} \dist(u,v)\] and an upper eccentricity \[\becc^\prime(B) = \max_{v \in V(B)} \ecc(v),\] in which case the block with the maximal upper eccentricity would be exactly a block containing a peripheral vertex, and the maximum such eccentricity would be equal to the traditional vertex diameter. We will not explore these concepts further as they would not contribute significantly to our investigations.

\subsection{Block eccentricity vs. eccentricity within the BC-tree}
We must take care not to confuse these concepts of eccentricity, radius, and diameter for blocks with their synonymous versions within the BC-tree itself. For example, if $G$ is any graph with two blocks $B_1$ and $B_2$ and articulation point $a$, then $\mathfrak{B_1}$ and $\mathfrak{B_2}$ will necessarily have the same vertex eccentricity in $\mathfrak{G}$, while $B_1$ and $B_2$ may have very different block eccentricities in $G$, depending on their relative sizes.

\section{Block eccentricity and centrality}\label{sec:ecc}

We now review the impact of separating structures on geodesics and distance measures before we arrive at our first main result in Theorem~\ref{thm:bcenters}, a characterization of type A and type B graphs in terms of the number and properties of their central points and central blocks as well as the relation between block and vertex radius. We will conclude with an application to self-centered graphs and a remark on the locations of eccentric vertices with respect to the central block.

\subsection{Separating structures and distance measures}

The rationale behind the name \emph{separating vertex} as a synonym for articulation point is that such a vertex $a$ may be said to \emph{separate} two other vertices $v_1$ and $v_2$ if all paths between $v_1$ and $v_2$ must pass through $a$; we may likewise say that one vertex or block \emph{separates} any two other vertices or blocks (although we will forbid the improper case of a block separating two non-articulation points inside it, or a non-articulation point inside it from anything else; this will prevent some obstreporous cases which would violate Proposition~\ref{prop:separating} and later results).

If a vertex or block is separating in this fashion, it has immediate consequences for distance measurements.

\begin{prop}\label{prop:separating}
If a vertex $a$ separates two other vertices $v_1$ and $v_2$, then:
\[\dist(v_1,v_2) = \dist(v_1,a) + \dist(a,v_2).\]
If $a$ separates two blocks $B_1$ and $B_2$:
\[\dist(B_1,B_2) = \dist(B_1,a) + \dist(a,B_2).\]
If $a$ separates vertex $v_1$ and block $B_2$: 
\[\dist(v_1,B_2) = \dist(v_1,a) + \dist(a,B_2).\]
If a block $B$ separates two vertices $v_1$ and $v_2$, and if $a_i$ is the articulation point in $B$ which separates it from $v_i$, then
\[\dist(v_1,v_2) = \dist(v_1,B) + \dist(a_1,a_2) + \dist(B,v_2).\]
If $B$ separates two other blocks $B_1$ and $B_2$ and $a_i$ is the articulation point in $B$ separating it from $B_i$, then
\[\dist(B_1,B_2) = \dist(B_1,B) + \dist(a_1,a_2) + \dist(B,B_2).\]
Finally, if $B$ separates vertex $v_1$ and block $B_2$ where $a_1$ and $a_2$ are the articulation points in $B$ separating it from $v_1$ and $B_2$ respectively, then
\[\dist(v_1,B_2) = \dist(v_1,B) + \dist(a_1,a_2) + \dist(B,B_2).\]
\end{prop}

\subsection{Central blocks vs. central points}

We being by exploring how our concept of central block relates to the traditional concept of central point. For context, we review the motivating 1953 result of Harary and Norman~\cite{harary1953dissimilarity} and rephrase the proof in terms of our concept of block eccentricity. (Note ``Husimi trees'' in the title of the original manuscript refers to cactus graphs, and the lemma statement calls a block a ``star.'')

\begin{lemma}\cite[Lemma 1]{harary1953dissimilarity}\label{lem:vcenters}
There is a block which contains all central points.
\end{lemma}

\begin{proof}
Say the graph $G$ has two central points $v_1$ and $v_2$ which are not in the same block. Then there must be some articulation point $a$ separating them. Consider the articulation components at $a$, and let $G_i$ be the component containing $v_i$ and $G_0$ be the graph union of the remaining components. In other words, $G_1$ is $a$ and everything that $a$ does not separate from $v_1$; $G_2$ is the same for $v_2$; and $G_0$ is $a$ and everything else.

By Proposition~\ref{prop:separating}, the vertices in $G_0$ and $G_1$ are strictly closer to $a$ than they are to $v_2$, so their distance to $a$ is less than $\ecc(v_2) = \rad(G)$; similarly, the vertices in $G_2$ (and again $G_0$) are closer to $a$ than to $v_1$, so their distance to $a$ also is less than $\ecc(v_1) = \rad(G)$. Thus, $a$ has a lower eccentricity than $\rad(G)$, which is a contradiction.
\end{proof}

This idea of a block which contains all the central points corresponds exactly to our new definition of a central block, and the two graphs types, A and B, correspond to the properties of the central block(s).

\begin{thm}\label{thm:bcenters}
A block is central iff it contains every central point. Further, the graph falls into one of two types:
\begin{enumerate}
\item \label{case:bcentersA} Type A: There is a unique central point which is an articulation point, all blocks containing it are central blocks, and \[\brad(G) = \rad(G).\]
\item \label{case:bcentersB} Type B: There is a unique central block $B$ which contains every central point, and \[\brad(G) < \rad(G).\]
\end{enumerate}
\end{thm}

\begin{ex}
Note that a type B graph may have a single central point, even one which is an articulation point: consider a 6-cycle with vertices labelled, in order, $v_1, v_2, \cdots, v_6$. Add a pendant at $v_2$ and $v_4$ and a length 2 path at $v_6$. Now $v_6$ is the unique center with vertex eccentricity 3, and the 6-cycle is the unique central block with block eccentricity 2. 
\end{ex}

See Corollary~\ref{cor:bcentersB} for a complementary result to Theorem~\ref{thm:bcenters}\eqref{case:bcentersB}, that
\[\rad(G) - \diam(B) \leq \brad(G).\]
It is important to note that a comparable result is not true for peripheral blocks: for example, for a 6-cycle and a 4-cycle joined at a single vertex, the 4-cycle is the only peripheral block (with block eccentricity 3, as opposed to the 6-cycle at 2), but the diameter is realized by a point from each cycle.

\begin{proof}[Proof of Theorem~\ref{thm:bcenters}]
First, we prove that every central block contains every central point by repeated applications of Proposition~\ref{prop:separating}. This will also imply that either there is a unique central block or else there are multiple central blocks intersecting at a unique central point.

Say there is a central block $B$ and a central point $v \not\in B$. Let $a$ be any articulation point separating $v$ from $B$ (in particular, $a \neq v$, although $a$ may be in $B$). Consider the articulation components at $a$: let $G_B$ be the articulation component containing $B$ (i.e., everything $a$ does not separate from $B$), let $G_v$ be the component containing $v$ (i.e., everything $a$ does not separate from $v$), and let $G_0$ be the union of the remaining components (or, if there are none, set $G_0 = \{a\}$). These subgraphs cover $G$ and intersect pairwise at $a$. 

Now let $u$ be an eccentric vertex of $a$, that is, \[\dist(a,u) = \ecc(a) \geq \rad(G).\] This vertex must be in $G_v$. If it were not, then $a$ would separate $u$ from $v$,so \[\dist(u,v)  = \dist(u,a) + \dist(a,v) \geq 1 + \rad(G),\] yet $\ecc(v)  = \rad(G)$ (note our assumptions implied $G$ is not an isolated vertex, so $u \neq a$).  Thus $u \not \in G_B$, or $a$ separates $u$ from $B$, so \[\dist(B,u) = \dist(B,a) + \dist(a,u) \geq \rad(G).\] We also know \[\dist(B,u) \leq \becc(B) = \brad(G) \leq \rad(G).\] Therefore, the inequalities are equalities, which means that $\dist(B,a) = 0$, so $a \in B$; that $\dist(a,u) = \rad(G)$, so $a$ is a central point; and that $\brad(G) = \rad(G)$.

These three conclusions are contradictory. Since $a$ was an arbitrary vertex separating $v$ from $B$, then $a \in B$ implies that $v$ must be in a block adjacent to $B$, call it $B_v$, and $B$ and $B_v$ share $a$ as an articulation point. If we calculate $\becc(B_v)$, it will be too small: Say $u \not \in G_v$, that is, $u$ is separated from the central point $v$ by $a$, so \[\dist(u,B_v) = \dist(u,a) < \dist(u,v) = \rad(G).\] Additionally, if $u \in G_v$, either $u \in B_v$ and $\dist(u,B_v) = 0$, or else $B_v$ separates $u$ from $a$ (since $a \in B_v$ and $a$ is not an articulation point in $G_v$), in other words, \[\dist(u,B_v) < \dist(u,a) = \rad(G).\] Therefore, $\ecc(B_v) < \rad(G)$, which is not compatible with $\brad(G) = \rad(G)$. 

To see that any block containing the center is central, and to check the remaining conditions, we consider two separate cases.

Say $G$ has multiple central blocks. Their intersection must contain the set of central points, so there must be a unique central point $a$ which is an articulation point. We need to show that $\brad(G) = \rad(G)$, so that the graph is type B, and also that every block containing $a$ is indeed a central block. Consider the set of articulation components at $a$, call them $\{G_i\}$, and let $B_i$ be the block in each $G_i$ which contains $a$. Then $a$ has some eccentric vertex $v$; without loss of generality, $v \in G_1$ and (since there are multiple central blocks) $B_2$ is a central block. Since $a$ separates $v$ from $B_2$, \[\dist(B_2,v) = \dist(a,v) = \rad(G),\] that is, $\brad(G) = \becc(B_2) \geq \rad(G)$. In fact, no $B_i$ has greater eccentricity: since $a$ is not an articulation point in any $G_i$, it is separated from $G_i - B_i$ by $B_i$, that is, \[\dist(u,B_i) < \dist(u,a) \leq \rad(G)\] for all $u \in G_i$, and $B_i$ is separated from any other articulation component by $a$, so \[\dist(u,B_i) = \dist(u,a) \leq \rad(G)\] for any $u \not \in G_i$.

Say $G$ has a unique central block $B$. We need only verify $\brad(G) < \rad(G)$. Let the articulation components of the graph at $B$ be called $\{B \cup G_i\}$, let $a_i$ be the articulation point shared by $B$ and $G_i$, and let $B_i$ be the block in each $G_i$ containing $a_i$. Recall Proposition~\ref{prop:br_v_r} that $\brad(G) \leq \rad(G)$ and assume $\brad(G) = \rad(G)$; then there is some vertex $v$ of distance $\rad(G)$ from $B$, say $v \in G_1$, in which case \[\dist(a_1,v) = \dist(B,v) = \rad(G).\] If $u$ is any vertex not in $G_1$, then $a_1$ separates $u$ from $v$, so \[\dist(u,v) > \dist(a_1,v) = \rad(G),\] and $u$ is not a central point. In particular, no vertex in $B$ other than $a$ may be central, but $B$ must contain a non-trivial center, so $a$ is a central vertex. Additionally, since $a_1 \in B_1$, Proposition~\ref{prop:ecc} says \[\ecc(B_1) \leq \ecc(a_1) = \rad(G),\] and $B_1$ is also a central block, which is a contradiction.
\end{proof}

\subsection{Self-centered graphs}
Theorem~\ref{thm:bcenters} completely characterizes graphs with maximal diameter of $2\rad$. As an immediate corollary, we can also find a necessary condition for a graph to have minimal diameter. These are known as \emph{self-centered graphs} because every vertex has the same eccentricity and so is a central vertex.
\begin{cor}
If $\rad(G)  = \diam(G)$, then $G$ is non-separable. 
\end{cor}

\subsection{Eccentric vertices}

We now investigate the implications of central blocks for eccentric vertices and geodesics realizing  eccentricity.

\begin{thm}\label{thm:whereismax}
Let $G$ be a graph.
\begin{enumerate}
\item \label{case:whereismaxA} If $G$ is type A with $a$ the central point and $\{G_i\}$ the articulation components at $a$, then any $v \in G_i$ has at least one eccentric vertex outside of $G_i$. In particular, $a$ itself has eccentric vertices in at least two different articulation components $G_i$.
\item \label{case:whereismaxB} If $G$ is type B with $B$ the central block and $\{G_i\}$ the articulation components at $B$, then any $v \in G_i$ has all its eccentric vertices outside $G_i$.
\end{enumerate}
In particular, any $v$ has an eccentric vertex $u$ so that the geodesic from $v$ to $u$ passes through at least a central block.
\end{thm}

\begin{ex}
Note that Theorem \ref{thm:whereismax}\eqref{case:whereismaxA} does not prevent $v$ from having eccentric vertices in its own $G_i$. Consider a square pyramid, i.e., a 4-cycle with another vertex at the peak joined to each of the cycle's vertices by an edge. Duplicate the pyramid and identify the peaks to make a vertex $a$: then $a$ is the unique center and an articulation point with 2 articulation components and eccentricity 1, realized by the other 4 vertices in each component. Any other vertex has eccentricity 2 with 5 eccentric vertices, all 4 vertices other than $a$ in the other articulation component and the non-adjacent vertex from its own component.
\end{ex}

\begin{ex}
Similarly, Theorem \ref{thm:whereismax}\eqref{case:whereismaxB} does not specify whether $v$'s eccentric vertices are in $B$ or another $G_i$. Consider a 4-cycle and a pair of 3-cycles attached at adjacent vertices. The 4-cycle is the central block $B$ with block eccentricity 1, and both of its articulation points are central points with eccentricity 2 (realized at the opposite vertex on the 4-cycle and at the pair of non-articulation points on the 3-cycle it does not touch). A non-articulation point on one of the 3-cycles has 3 eccentric vertices, one on the 4-cycle and the other 2 on the other 3-cycle.
\end{ex}

\begin{proof}[Proof of Theorem~\ref{thm:whereismax}]
Say $G$ is type A. Let $B_i$ be the block in $G_i$ which contains the central point $a$. By Theorem~\ref{thm:bcenters}\eqref{case:bcentersA}, $B_i$ is central and has eccentricity $\rad(G)$; since it is strictly closer to all vertices in $G_i$ than $a$ is, its block eccentricity is not realized in $G_i$. That is, for any chosen $G_i$, there is some vertex $v_j$ in another $G_j$ of distance $\rad(G)$ from $B_i$, i.e., $\dist(a,v_j) = \rad(G)$. Now consider any vertex $u_i \in G_i$: its distance to any $v_i \in G_i$ is 
\[\dist(u_i,v_i) \leq \dist(u_i, a_i) + \dist(a_i, v_i) \leq \dist(u_i, a_i) + \rad(G);\]
however, it is separated from $v_j$ by $a$, and so its distance to $v_j$ is
\[\dist(u_i,v_j) = \dist(u_i,a) + \dist(a,v_j) = \dist(u_i,a) + \rad(G).\]
In other words, while the eccentricity of $u_i$ may be realized inside its own $G_i$, it is definitely realized in another $G_j$.

Say $G$ is type B, and let $a_i$ be the articulation point shared by $G_i$ and the central block $B$. Theorem~\ref{thm:bcenters}\eqref{case:bcentersB} says $\becc(B) < \rad(G)$, but $\ecc(a_i) \geq \rad(G)$, so the eccentricity of any $a_i$ is realized in $B$ or another $G_j$, i.e., there is some $v_j \in G_j$ with $\dist(a_i,v_j) \geq \rad(G)$. For any other $u \in G_i$, note that $v_i \in G_i$ obeys
\[\dist(u,v_i) \leq \dist(u,a_i) + \dist(a_i,v_i) < dist(u,a_i) + \rad(G)\]
while $v_j$ obeys
\[\dist(u,v_j) = \dist(u,a_i) + \dist(a_i,v_j) \geq \dist(u,a_i) + \rad(G).\]
That is, the eccentricity of $u$ is not realized in $G_i$.
\end{proof}

\section{Application: Bounding radius and diameter}\label{sec:bound}

Recall the classical bound 
\[\rad(G) \leq \diam(G) \leq 2\rad(G).\]
A cycle with an even number of edges realizes the lower bound, a path with an even number of edges realizes the upper bound, and the intermediate values are realized by, for example, the minimal order graphs of Ostrand~\cite{ostrand1973graphs}. While the classical bound is sharp, it still leaves some opportunities for further refinement; for example, we verify that graphs of type A all satisfy the upper bound, and graphs of type B may be given a sharper range using the diameter of the central block.

\begin{thm}\label{thm:r_vs_d}
If $G$ is type A,
\[\diam(G) = 2\rad(G).\]
If $G$ is type B with central block $B$,
\[2\rad(G) - \diam(B) \leq \diam(G) \leq 2\rad(G).\]
\end{thm}

The lower bound is an improvement of the classical one when $\diam(B) < \rad(G)$, which is not rare in graphs with a significant number of articulation points or where several blocks are of comparable size, for example, in the wedge of two cycles of similar (but not the same) size. 

The proof below also demonstrates a new upper bound,
\[\diam(G) \leq 2\brad(G) + \diam(B),\]
which is an improvement upon the classical one when $\frac{1}{2}\diam(B) < \rad(G) - \brad(G)$, which is rarer but still possible. Let us examine a situation where the lower bound is sharp, which will motivate the proof for the general case:
\begin{ex}\label{ex:sharp}Let $G$ be a graph consisting of a $2n$-cycle $B$ with a path of length $l$ attached at each of its vertices. Then the diameter is realized by the endpoints of two paths attached at opposite points on the cycle, that is, $\diam(G) = 2l + n.$ Similarly, each vertex on the cycle itself is central, and the most distant vertex is the endpoint of the path immediately opposite it, so $\rad(G) = l + n.$ Since $\diam(B) = n$, $\diam(G) = 2\rad(G) - \diam(B).$
\end{ex}

\begin{proof}[Proof of Theorem \ref{thm:r_vs_d}]
For type A graphs, if $a$ is the unique central point and $G_i$ are the articulation components at $a$ with $B_i$ the block in $G_i$ containing $a$, then we verified in Theorem~\ref{thm:bcenters} that $\brad(G) = \rad(G)$ and the $B_i$ are central, i.e., $\becc(B_i) = \ecc(a)$, and there are at least two $G_i$ with eccentric vertices of $a$, say $v_1 \in G_1$ and $v_2 \in G_2$. Then $\dist(v_1, v_2) = 2\rad(G)$. 

For type B, for the lower bound, let $B$ be the central block with articulation points $a_i$ and articulation components $G_i$ (where $a_i \in G_i$). If $\diam(B) \geq \rad(G)$, then $2\rad(G) - \diam(B) \leq \rad(G)$, and the bound is no stronger than the classical one. Assume $\diam(B) < \rad(G)$, in which case no vertex in $B$ can have an eccentric vertex in $B$. Select any vertex $v \in B$ and select one of its eccentric vertices, without loss of generality, $u_1 \in G_1$. Then
\[\dist(a_1,u_1) = \dist(v,u_1) - \dist(v,a_1) \geq \rad(G) - \diam(B).\]
Now $a_1$ has its own eccentric vertices, including one we will name $u_2$, which is not in $G_1$ by Theorem~\ref{thm:whereismax}. Then
\[\dist(u_1, u_2) = \dist(u_1,a_1) + \dist(a_1,u_2) \geq 2\rad(G) - \diam(B).\]
The inequality is attained, as shown by Example~\ref{ex:sharp}.

For the upper bound on type B, consider two vertices of maximal distance, say $u_1$ and $u_2$. By Theorem~\ref{thm:whereismax}, they cannot be located in the same $G_i$, i.e., any path between them must pass through $B$. Say $u_1 \in G_1$ and $u_2 \in G_2$: then 
\[\dist(u_1, u_2) = \dist(u_1, a_1) + \dist(a_1, a_2) + \dist(a_2, u_2) \leq 2\becc(B) + \diam(B).\]
If $u_1 \in G_1$ and $u_2 \in B$, the equation is even simpler:
\[\dist(u_1, u_2) = \dist(u_1, a_1) + \dist(a_1, u_2) \leq \becc(B) + \diam(B).\]
\end{proof}

This theorem allows us to prove a complementary lower bound to Theorem~\ref{thm:bcenters}\eqref{case:bcentersB}.

\begin{cor}\label{cor:bcentersB} For a graph $G$ of type B with central block $B$,
\[\rad(G) - \brad(G) \leq \diam(B).\]
\end{cor}

\begin{proof}
Theorem~\ref{thm:r_vs_d} says
\[2\rad(G) - \diam(B) \leq 2\brad(G) + \diam(B),\]
and the result follows.
\end{proof}

\section{Application: A subgraph realizing maximal eccentricity}\label{sec:subgraph}

We propose a subgraph which can be used to study both the vertex and block eccentricity. As suggested by the proofs of Lemma~\ref{lem:vcenters} and Theorem~\ref{thm:bcenters}, much of the block structure of a graph is redundant: there is a unique central block or central articulation point), and the graph branches out in articulation components from this central block/vertex. We may reduce the complexity of the graph by pruning down the number and size of such components.

\begin{const}\label{const:big} Let $G$ be a graph.
\begin{enumerate}
\item\label{case:bigA} Say $G$ is type A with $a$ the central point and $\{G_i\}$ the articulation components at $a$. For each $i$, identify a vertex in $G_i$ of maximal distance from $a$. Retain a geodesic from this vertex to $a$, along with all blocks from which it takes at least one edge. Delete the rest of $G_i$.
\item\label{case:bigB} Say $G$ is type B with $B$ the central block and $\{G_i\}$ the articulation components at $B$. For each $i$, identify a vertex in $G_i$ of maximal distance from $B$. Retain a geodesic from this vertex to $B$, along with all blocks from which it takes at least one edge. Delete the rest of $G_i$.
\end{enumerate}
\end{const}

\begin{thm}\label{thm:replacement}
Let $G$ be a connected graph and $G'$ be the altered graph given by Construction~\ref{const:big}:
\begin{enumerate}
\item $G'$ is a connected subgraph of $G$.
\item $\bct{G'}$ is a subgraph of $\bct{G}$.
\item $\bct{G'}$ is a path or starlike tree.
\item $G'$ has the same central points and blocks and the same vertex and block radius as $G$.
\item The peripheral vertices of $G'$ are a subset of the peripheral vertices of $G$, and it has the same vertex diameter as $G$.
\end{enumerate}
\end{thm}

Recall that a \emph{starlike} tree has exactly one vertex of degree greater than 2.

\begin{proof}
We may examine the construction in the BC-tree rather than in $G$: for a graph of type A (respectively, type B), we pick a vertex $v$ in $G_i$ of maximal distance from $a$ (resp., $B$), say $v$ is in a block called $D$, and find a path in $\bct{G}$ from $\bct{D}$ to $\bct{a}$ (resp., $\bct{B}$). Alter $\bct{G}$ one step at a time by deleting leaves and their accompanying pendant edges, reflecting each deletion in $G$ as we go by removing a block. Deleting a leaf $\bct{D}$ in $\bct{G}$ will correspond to deleting a block $D$ in $G$ which was connected to a single articulation point $a_D$, which will reduce the size and order of the graph but not disconnect it. That is, $G'$ is connected, and $G'$ and $\bct{G'}$ are subgraphs of $G$ and $\bct{G}$, respectively. Additionally, as long as $\deg{\bct{a_D}} \geq 3$, removing $D$ reduces the number of blocks intersecting at $a_D$ but leaves it an articulation point, so $\bct{G'}$ is the BC-tree of $G'$. On the other hand, if $\deg{\bct{a_D}} = 2$, then $a_D$ is an articulation point between two blocks, and removing $D$ renders it a non-articulation point and converts $\bct{a_D}$ to a leaf itself, so $\bct{G'}$ is not a BC-tree of anything until we also remove $\bct{a_D}$ and its pendant edge.

By construction, all central blocks and so also all central points are retained in $G'$. Additionally, any $u$ and $v$ which survive to $G'$ will have the same distance there which they did in $G$: removing blocks which correspond to leaves in the BC-tree will not interfere with geodesics between surviving vertices or affect the relative distance between them; in particular, the eccentricity of any vertex or block will not increase from $G$ to $G'$. Similarly, the vertex eccentricity of $a$ or the block eccentricity of $B$ with respect to each $G_i$ will remain the same.

We next show that the radius does not change. If $G$ is type A, then the eccentricity of $a$ will not change: if $v_i$ is a vertex in $G_i$ of maximal distance from $a$, then either $v_i$ or some other vertex of equal distance in $G_i$ will survive to $G'$. Additionally, if vertex $v \in G_i$ survives to $G'$, then its eccentricity also does not change by Theorem~\ref{thm:whereismax} since it is realized by at least one vertex in some other $G_j$, and this vertex is separated from $v$ by $a$, so either it or another vertex of equal distance in $G_j$ will survive.

Similarly, if $G$ is type B, then the eccentricities of all the vertices in $B$ will also remain the same because the distance between vertices within $B$ will be unchanged, and the eccentricity of $B$ and so all its vertices with respect to any given $G_i$ will be the same from $G$ to $G'$. Additionally, the eccentricity of no vertex outside $B$ will change: by Theorem~\ref{thm:whereismax}, the eccentricity of a vertex in $G_i$ is attained only by some vertex from which it is separated by $a_i$, and that vertex is either in $B$ (so will survive) or is in some other $G_j$ (in which case it or another vertex of the same distance will survive).

We finally show that the diameter does not change. A type A graph remains type A because its central blocks remain central, and its radius is preserved, so its diameter is also preserved. For a type B graph, let $u$ and $v$ be vertices of maximal distance. Their geodesic must pass through $B$ by Theorem~\ref{thm:whereismax}, and any portion of it in $B$ is preserved. If it contains any portion in some $G_i$, then it passes through some $a_i$, and, even if that portion does not itself survive, some other geodesic in $G_i$ of the same length will. 
\end{proof}

Note that this construction could be optimized further. For example, if we were interested exclusively in diameter or in type A graphs, we could select two vertices of maximal distance in different $G_i$ and preserve only the path between them (which would retain the central block if type B and central point if type A); even if we wish only to study radius in type B graphs, we may still first select a set of all eccentric points of the central points along with geodesics back to their central points and retain the blocks with edges in these geodesics.

\section{Application: The Randi\'c index of a cactus graph}\label{sec:randic}%mention Gutman's book here!

We now turn towards an invariant from mathematical chemistry which is used to study the branching of a graph. It was originally defined by Milan Randi\'c in 1975 in an attempt to mathematically characterize branching in a way consist with boiling point and other structure-related properties such as enthalpy of formation of alkanes and the relationship of vapor pressure to temperature~\cite{randic1975characterization}. It has been experimentally verified to be associated to boiling and reactivity of hydrocarbons and has now become a standard tool for evaluating molecular structure in quantitative structure-activity relationship (QSAR) models, that is, regressive models that predict biological activity, physicochemical properties, and toxicological responses of chemical compounds based on their molecular structure (see, for example, \cite{kier1986molecular, pogliani2000molecular, garcia2008some, todeschini2008handbook, kier2012molecular}). 

Randi\'c originally formulated his invariant in terms of the graph adjacency matrix, although the the common formulation today is due to Balaban~\cite{balaban1982discriminating}. First, define the \emph{weight} of an edge $[u,v]$ to be:
\[w[u,v] = \frac{1}{\sqrt {\deg(u) \deg{v}}}\]
and then define the Randi\'c index of $G$ as:
\[\Rand(G) = \sum_{[u,v] \in E(G)} w[u,v]\]
although an alternate formulation due to Caparorossi, Gurman, Hansen, and Pavlovic in 2003 is~\cite{caporossi2003graphs}:
\[\Rand(G) = \frac{\#V(G)-n_0}{2} - \sum_{[u,v]\in E(G)} w^\ast[u,v]\]
where $n_0$ is the number of isolated vertices and $w^\ast[u,v]$ is a measure of asymmetry of edge weights, that is,
\[w^\ast[u,v] = \frac{1}{2}\left(\frac{1}{\sqrt{\deg{u}}} - \frac{1}{\sqrt{\deg{v}}}\right)^2,\]
which is sometimes easier to engage with as it shows, e.g., that $\Rand(G) \leq \frac{\#V(G)}{2}.$

The Randi\'c index is worth studying from a graph theoretic point of view, in particular because it identifies a type of branching not easily encapsulated by other invariants. The Graffiti computer prediction program first identified a possible link to graph radius, although it has been resistant to repeated efforts to prove it. There have since been similar investigations into diameter: 
\begin{conj}\label{conj}
Let $G$ be a graph.
\begin{enumerate}
\item \cite{fajtlowicz1988graffiti} If $G$ is not a path with an even number of vertices, $\Rand(G) \geq \rad(G).$
\item \cite{aouchiche2007variable} $\Rand(G) \geq \diam(G) + \sqrt 2 - \frac{\#V(G) + 1}{2}$.% and $\frac{\Rand}{\diam} \geq \frac{\#V(G)-3+2\sqrt 2}{2\#V(G)-2}$.
\end{enumerate}
\end{conj}

\begin{ex}\label{ex:basicr} For example, the path on $2$ vertices obeys $\diam=\rad=\Rand=1.$ For path $P_n$ on $n > 2$ vertices, $\diam=n-1,$ $\rad = \left\lfloor \frac{n}{2} \right\rfloor,$ and $\Rand=\frac{n-3}{2} + \sqrt 2.$ In either case, we have a clear relationship between graph size and either Randi\'c index or extremal eccentricities, which we may translate into a relationship between Randi\'c index and the eccentricities themselves:
\[\Rand(P_n) = \begin{cases} \rad(P_n) + \sqrt 2 - \frac{3}{2} & if~n~even\\ \rad(P_n) + \sqrt 2 - 1& if~n~odd\end{cases}\]
and
\[\Rand(P_n) = \diam(P_n) + \sqrt 2 - \frac{n+1}{2}.\]
For the cycle on $n$ vertices, $C_n$, $\diam = \rad = \left\lfloor \frac{n}{2} \right\rfloor$ and $\Rand = \frac{n}{2},$ so 
\[\Rand(C_n) \geq \rad(C_n) = \diam(C_n).\]
\end{ex}

In keeping with the relevance of this invariant to mathematical chemistry, these results are known by work of Cygan-Pilipczuk-\v{S}krekovski and Stevanovi\'c for \emph{chemical graphs}, those of degree at most 4, a necessary condition for a graph to model a molecular structure. 

\begin{thm}\label{thm:base} Let $G$ be a connected chemical graph.
\begin{enumerate}
\item \cite[Theorem 1.3]{cygan2012inequality}\label{biblio:chemical_rad} If $G$ is not an even path, \[\Rand(G) \geq \rad(G).\]
\item \cite[Theorem 5]{stevanovic2008randic}\label{biblio:chemical_diam} If $\#V(G) \geq 3$, \[\Rand(G) \geq \diam(G) + \sqrt{2} - \frac{\#V(G) + 1}{2}.\]
\end{enumerate}
\end{thm}

We will apply Construction~\ref{const:big} to a generic cactus graph and then induct from the base case of chemical graphs, so we will first need a lemma about deleting blocks:

\begin{lemma}\label{lem:wedge}
Let $G$ be a graph with $a$ an articulation point and $\{G_i\}$ the articulation components at $a$. Then: 
\[R(G) \geq \sum \Rand(G_i) + \sqrt{\deg{a}} - \sum \sqrt{\deg{a_i}}.\]
In particular, $\Rand(G) \geq \Rand(G_i)$ for any $i$.
\end{lemma}

\begin{proof}
Consider $G_i$ as a graph and not a subgraph of $G$. Let $a_i$ be the copy of $a$ in $G_i$ and $N(a_i)$ all its adjacent vertices. All edges have the same weight in $G$ and $G_i$ except for those incident to $a_i$. Let $S(i)$ be the sum of the weights of all edges in $G_i$ incident to $a$, or
\[S(i) = \sum_{w \in N(a_i)} \frac{1}{\sqrt{\deg{w}}\sqrt{\deg{a_i}}} \leq \sum_{w \in N(a_i)}\frac{1}{\sqrt{\deg{a_i}}} = \sqrt{\deg{a_i}}.\]
Now:
\begin{multline*}
\Rand(G) - \sum \Rand(G_i) = \sum -S(i) + S(i)\frac{\sqrt{\deg{a_i}}}{\sqrt{\deg{a}}}\\ \geq \sum -\sqrt{\deg{a_i}} + \frac{\deg{a_i}}{\sqrt{\deg{a}}}
= \sqrt{\deg{a}} - \sum \sqrt{\deg{a_i}}.
\end{multline*}

Finally, a result of Bollob\'as and Erd\H{o}s \cite[Theorem 3]{bollobas1998extremal} shows the minimal Randi\'c index on a fixed size graph is realized at the star, or $\Rand(G) \geq \sqrt{\#V(G)-1}$, so $\Rand(G_i) \geq \sqrt{\deg{a_i}}$.
\end{proof}

As an immediate consequence of Theorem~\ref{thm:replacement}, we have

\begin{thm}\label{thm:result}
If $G$ is a cactus,
\[\begin{split}
\Rand(G) &\geq \rad(G) + \sqrt {2} - \frac{3}{2}.\\
\Rand(G) & \geq \diam(G) + \sqrt{2} -  \frac{\#V(G) + 1}{2}.
\end{split}\]
In particular, Conjecture~\ref{conj} is true for cactus graphs.
\end{thm}

\begin{proof}
Consider a subgraph $G'$ given by Construction~\ref{const:big}. Its BC-tree is starlike or a tree. Since $G$ is a cactus, any block is either a cycle or a bridge, and each articulation point connects exactly two blocks, so an articulation point between bridges has degree 2; between a bridge and a cycle degree 3; between two cycles degree 4. Therefore, $G'$ is a chemical graph, and the result holds for $G'$ by the work of  referenced in Theorem~\ref{thm:base}.

By Lemma~\ref{lem:wedge}, $G'$ has smaller Randi\'c index than $G$; by Theorem~\ref{thm:replacement}, it has the same radius and diameter. 
\end{proof}

\bibliographystyle{alpha}
\bibliography{/Users/mdo55093/Documents/math/biblio-MASTER}

\begin{thebibliography}{GDGdJOP08}

\bibitem[AHZ07]{aouchiche2007variable}
Mustapha Aouchiche, Pierre Hansen, and Maolin Zheng.
\newblock Variable neighborhood search for extremal graphs: 19: {F}urther
  conjectures and results about the {R}andi{\'{c}} index.
\newblock {\em MATCH Commun. Math. Comput. Chem.}, 58(1):83, 2007.

\bibitem[Bal82]{balaban1982discriminating}
Alexandru~T Balaban.
\newblock Highly discriminating distance-based topological index.
\newblock {\em Chem. Phys. Lett.}, 89(5):399--404, 1982.

\bibitem[BE98]{bollobas1998extremal}
B{\'{e}}la Bollob{\'{a}}s and Paul Erd{\H{o}}s.
\newblock Graphs of extremal weights.
\newblock {\em Ars Combin.}, 50:225--233, 1998.

\bibitem[BH90]{buckley1990distance}
Fred Buckley and Frank Harary.
\newblock Distance in graphs.
\newblock {\em (No Title)}, 1990.

\bibitem[CGHP03]{caporossi2003graphs}
Gilles Caporossi, Ivan Gutman, Pierre Hansen, and Ljiljana Pavlovi{\'{c}}.
\newblock Graphs with maximum connectivity index.
\newblock {\em Comput. Biol. and Chem.}, 27(1):85--90, 2003.

\bibitem[CP{\v{S}}12]{cygan2012inequality}
Marek Cygan, Micha{\l{}} Pilipczuk, and Riste {\v{S}}krekovski.
\newblock On the inequality between radius and {R}andi{\'{c}} index for graphs.
\newblock {\em MATCH Commun. Math. Comput. Chem.}, 67(2):451--466, 2012.

\bibitem[Faj88]{fajtlowicz1988graffiti}
Siemion Fajtlowicz.
\newblock On conjectures of {G}raffiti.
\newblock {\em Discrete Math.}, 72(1-3):113--118, 1988.

\bibitem[GDGdJOP08]{garcia2008some}
Ram{\'{o}}n Garc{\'{i}}a-Domenech, Jorge G{\'{a}}lvez, Jesus~V.
  de~Juli{\'{a}}n-Ortiz, and Lionello Pogliani.
\newblock Some new trends in chemical graph theory.
\newblock {\em Chem. Rev.}, 108(3):1127--1169, 2008.

\bibitem[HN53]{harary1953dissimilarity}
Frank Harary and Robert~Z Norman.
\newblock The dissimilarity characteristic of {H}usimi trees.
\newblock {\em Annals of Mathematics}, pages 134--141, 1953.

\bibitem[Jor69]{jordan1869assemblages}
Camille Jordan.
\newblock Sur les assemblages de lignes.
\newblock {\em J. Reine Angew. Math.}, 70:185--190, 1869.

\bibitem[KH76]{kier2012molecular}
Lemont~B. Kier and Lowell~H. Hall.
\newblock {\em Molecular {C}onnectivity in {C}hemistry and {D}rug {R}esearch}.
\newblock Academic Press, New York, 1976.

\bibitem[KH86]{kier1986molecular}
Lemont~Burwell Kier and Lowell~H. Hall.
\newblock {\em Molecular {C}onnectivity in {S}tructure-{A}ctivity {A}nalysis}.
\newblock Research Studies Press, Letchworth, Hertfordshire, England, 1986.

\bibitem[Ost73]{ostrand1973graphs}
Phillip~A. Ostrand.
\newblock Graphs with specified radius and diameter.
\newblock {\em Discrete Math.}, 4:71--75, 1973.

\bibitem[Pog00]{pogliani2000molecular}
Lionello Pogliani.
\newblock From molecular connectivity indices to semiempirical connectivity
  terms: {R}ecent trends in graph theoretical descriptors.
\newblock {\em Chem. Rev.}, 100(10):3827--3858, 2000.

\bibitem[Ran75]{randic1975characterization}
Milan Randi{\'{c}}.
\newblock Characterization of molecular branching.
\newblock {\em J. Am. Chem. Soc.}, 97(23):6609--6615, 1975.

\bibitem[Ste08]{stevanovic2008randic}
Dragan Stevanovi{\'{c}}.
\newblock On the {R}andi{\'{c}} index and diameter of chemical graphs.
\newblock {\em Recent Results in the Theory of Randi{\'{c}} Index, Univ.
  Kragujevac, Kragujevac}, pages 49--55, 2008.

\bibitem[TC08]{todeschini2008handbook}
Roberto Todeschini and Viviana Consonni.
\newblock {\em Handbook of molecular descriptors}, volume~11.
\newblock John Wiley \& Sons, 2008.

\end{thebibliography}

\end{document}